\documentclass{amsart}
\usepackage{enumerate}

\theoremstyle{plain}
\newtheorem{theorem}{Theorem}[section]
\newtheorem{lemma}[theorem]{Lemma}

\theoremstyle{definition}
\newtheorem{definition}[theorem]{Definition}
\newtheorem{example}[theorem]{Example}

\theoremstyle{remark}
\newtheorem{remark}[theorem]{Remark}

\numberwithin{equation}{section}

\title{Starlikeness associated with lemniscate of Bernoulli}

\author[V. Madaan]{Vibha Madaan}
\address{Vibha Madaan, Department of Mathematics, University of Delhi, Delhi--110 007, India}
\email{vibhamadaan47@gmail.com}

\author[A. Kumar]{Ajay Kumar$^\ast$}
\address{Ajay Kumar, Department of Mathematics, University of Delhi, Delhi--110 007, India}
\email{akumar@maths.du.ac.in}

\author[V. Ravichandran]{V. Ravichandran}
\address{V. Ravichandran, Department of Mathematics, National Institute of Technology, Tiruchirappalli--620015, India}
	\email{vravi68@gmail.com}

\begin{document}
	\thanks{$^\ast$Corresponding author, E-mail address: akumar@maths.du.ac.in}

	\begin{abstract}
		For an analytic function $f$ on the unit disk $\mathbb{D}=\{z:|z|<1\}$ satisfying $f(0)=0=f'(0)-1,$ we obtain sufficient conditions so that $f$ satisfies $|(zf'(z)/f(z))^2-1|<1.$ The technique of differential subordination of first or second order is used. The admissibility conditions for lemniscate of Bernoulli are derived and employed in order to prove the main results.
	
\end{abstract}
\maketitle

\section{Introduction}
The set of analytic functions $f$ on the unit disk $\mathbb{D}=\{z:|z|<1\}$ normalized as $f(0)=0$ and $f'(0)=1$ will be denoted by $\mathcal{A}$ and $\mathcal{S}$ be the subclass of $\mathcal{A}$ consisting of univalent functions.
A function $f\in\mathcal{SL}$ if $zf'(z)/f(z)$ lies in the region bounded by the right half of lemniscate of Bernoulli given by $\{w:|w^2-1|=1\}$ and such a function will be called \textit{lemniscate starlike}. Evidently, the functions in class $\mathcal{SL}$ are univalent and starlike i.e. $\operatorname{Re}(zf'(z)/f(z))>0$ in $\mathbb{D}.$ The set $\mathcal{H}[a,n]$ consists of analytic functions $f$ having Taylor series expansion of the form $f(z)=a+a_nz^n+a_{n+1}z^{n+1}+\ldots$ with $\mathcal{H}_1:=\mathcal{H}[1,1].$ For two analytic functions $f$ and $g$ on $\mathbb{D},$ the function $f$ is said to be \textit{subordinate} to the function $g,$ written as $f(z)\prec g(z)$ (or $f\prec g$), if there is a Schwarz function $w$ with $w(0)=0$ and $|w(z)|<1$ such that $f(z)=g(w(z)).$ If $g$ is a univalent function, then $f(z)\prec g(z)$ if and only if $f(0)=g(0)$ and $f(\mathbb{D})\subset g(\mathbb{D}).$ In terms of subordination, a function $f\in\mathcal{A}$ is   lemniscate starlike if $zf'(z)/f(z)\prec\sqrt{1+z}.$ The class $\mathcal{SL}$ was introduced by Sok\'{o}l and Stankiewicz \cite{MR1473947}.
	
The class $\mathcal{S}^*(\varphi)$ of \textit{Ma-Minda starlike functions} \cite{MR1343506}  is defined by
\[\mathcal{S}^*(\varphi):=\left\{f\in\mathcal{S}:\dfrac{zf'(z)}{f(z)}\prec\varphi(z)\right\}, \]	where $\varphi$ is analytic and univalent on $\mathbb{D}$ such that $\varphi(\mathbb{D})$ is starlike with respect to $\varphi(0)=1$ and is symmetric about the real axis with $\varphi'(0)>0.$ For particular choices of $\varphi,$ we have well known subclasses of starlike functions like for $\varphi(z):=\sqrt{1+z},\ \mathcal{S}^*(\varphi):=\mathcal{SL}.$ If $\varphi(z):=(1+Az)/(1+Bz),$ where $-1\leq B<A\leq1$, the class $\mathcal{S}^*[A,B]:=\mathcal{S}^*((1+Az)/(1+Bz))$ is called the class of \textit{Janowski starlike functions} \cite{MR0267103}. If for $0\leq\alpha<1,\ A=1-2\alpha$ and $B=-1,$ then we obtain $\mathcal{S}^*(\alpha):=\mathcal{S}^*[1-2\alpha,-1],$ the class of \textit{starlike functions of order $\alpha.$} The class $\mathcal{S}^*(\alpha)$ was introduced by Robertson \cite{MR0783568}. The class $\mathcal{S}^*:=\mathcal{S}^*(0)$ is simply the class of \textit{starlike functions}. If the function $\varphi_{PAR}:\mathbb{D}\to\mathbb{C}$ is given by
\[ \varphi_{PAR}(z):=1+\frac{2}{\pi^2}\left(\log\frac{1+\sqrt{z}}{1-\sqrt{z}}\right)^2,\ \operatorname{Im}\sqrt{z}\geq0  \]
then $\varphi_{PAR}(\mathbb{D}):=\{w=u+iv:v^2<2u-1\}=\{w:\operatorname{Re}w>|w-1|\}.$ Then the class $\mathcal{S}_P:=\mathcal{S}^*(\varphi_{PAR})$ of parabolic functions, introduced by R\o{}nning\cite{MR1128729}, consists of the functions $f\in\mathcal{A}$ satisfying
\[ \operatorname{Re}\left(\frac{zf'(z)}{f(z)}\right)>\left|\frac{zf'(z)}{f(z)}-1\right|,\ z\in\mathbb{D}.  \]
Sharma \emph{et. al} \cite{MR3536076} introduced the set $\mathcal{S}^*_C:=\mathcal{S}^*(1+4z/3+2z^2/3)$ which consists of functions $f\in\mathcal{A}$ such that $zf'(z)/f(z)$ lies in the region bounded by the cardioid
\[ \Omega_C:=\{w=u+iv:(9u^2+9v^2-18u+5)-16(9u^2+9v^2-6u+1)=0   \}.  \] The class $\mathcal{S}^*_e:=\mathcal{S}^*(e^z),$ introduced by Mendiratta \emph{et. al} \cite{MR3394060}, contains functions $f\in\mathcal{A}$ that satisfy $|\log (zf'(z)/f(z)|<1.$

For $b\geq1/2$ and $a\geq 1,$ Paprocki and Sok\'{o}l \cite{MR1473960} introduced a more general class $\mathcal{S}^*[a,b]$ for the functions $f\in\mathcal{A}$ satisfying $|(zf'(z)/f(z))^a-b|<b.$ Evidently, the class $\mathcal{SL}:=\mathcal{S}^*[2,1].$ Kanas \cite{MR2209585} used the method of differential subordination to find conditions for the functions to map the unit disk onto region bounded by parabolas and hyperbolas.
Ali \emph{et al.} \cite{MR2917253} studied the class $\mathcal{SL}$ with the help of differential subordination and obtained some lower bound on $\beta$ such that $p(z)\prec\sqrt{1+z}$ whenever $1+\beta zp'(z)/p^n(z)\prec\sqrt{1+z}\ (n=0,1,2),$ where $p$ is analytic on $\mathbb{D}$ with $p(0)=1.$ Kumar \emph{et al.} \cite{MR3063215} proved that whenever $\beta>0,\ p(z)+\beta zp'(z)/p^n(z)\prec\sqrt{1+z}\ (n=0,1,2)$ implies $p(z)\prec\sqrt{1+z}$ for $p$ as mentioned above.
	
Motivated by work in \cite{MR2917253,MR2209585,MR3496681,MR3063215,MR3394060,MR1128729,MR3536076,MR3518217}, the method of differential subordination of first and second order has been used to obtain sufficient conditions for the function $f\in\mathcal{A}$ to belong to class $\mathcal{SL}.$ Let $p$ be an analytic function in $\mathbb{D}$ with $p(0)=1.$ In Section \ref{fods}, using the first order differential subordination, conditions on complex number $\beta$ are determined so that $p(z)\prec\sqrt{1+z}$ whenever $p(z)+\beta zp'(z)/p^n(z)\prec\sqrt{1+z}\ (n=3,4)$ or whenever $p^2(z)+\beta zp'(z)/p^n(z)\prec 1+z\ (n=-1,0,1,2)$ and alike. Also, conditions on $\beta$ and $\gamma$ are obtained that enable $p^2(z)+zp'(z)/(\beta p(z)+\gamma)\prec 1+z$ imply $p(z)\prec\sqrt{1+z}.$
Section \ref{sods} deals with obtaining sufficient conditions on $\beta$ and $\gamma,$ using the method of differential subordination which implies $p(z)\prec\sqrt{1+z}$ if $\gamma zp'(z)+\beta z^2p''(z)\prec z/(8\sqrt{2})$ and others.
Section \ref{al} admits alternate proofs for the results proved in \cite{MR2917253} and \cite{MR3063215}. The proofs are based on properties of admissible functions formulated by Miller and Mocano \cite{MR1760285}.
	
	
\section{The admissibility condition}
 
Let $\mathcal{Q}$ be the set of functions $q$ that are analytic and injective on $\overline{\mathbb{D}} \setminus \mathbf{E}(q),$ where
		\[ \mathbf{E}(q)=\left\{ \zeta \in \partial \mathbb{D} :  \underset{z \rightarrow \zeta} \lim q(z)=\infty \right\}    \]
		and are such that $q'(\zeta) \neq 0$ for $\zeta \in \partial \mathbb{D}\setminus\mathbf{E}(q)$.
 
	\begin{definition}
		Let $\Omega$ be a set in $\mathbb{C} , q \in \mathcal{Q}$ and $n$ be a positive integer. The class of admissible functions $\Psi_n[\Omega,q],$ consists of those functions $\psi: \mathbb{C}^3 \times \mathbb{D} \to \mathbb{C}$ that satisfy the admissiblity condition $\psi(r,s,t;z)\not\in\Omega$ whenever $r=q(\zeta)$ is finite, $s =m\zeta q'(\zeta)$ and $\operatorname{Re} \left( \dfrac{t}{s}+1 \right) \geq m \operatorname{Re} \left( \dfrac{\zeta q''(\zeta)}{q'(\zeta)}+1 \right),$ for $z\in\mathbb{D},\zeta\in\partial\mathbb{D}\setminus E(q)$ and $m\geq n\geq 1.$  The class $\Psi_1[\Omega,q]$ will be denoted by $\Psi[\Omega,q].$
	\end{definition}

	\begin{theorem}\cite[Theorem 2.3b, p.\,28]{MR1760285}\label{main thm}
		Let $\psi\in\Psi_n[\Omega,q]$ with $q(0)=a.$ Thus for $p\in\mathcal{H}[a,n]$ such that \begin{equation}\label{adm basic}
		\psi(p(z),zp'(z),z^2p''(z);z)\in\Omega\ \ \Rightarrow \ \ p(z)\prec q(z).
		\end{equation}	
\end{theorem}
If $\Omega$ is a simply connected region which is not the whole complex plane, then there is a conformal mapping $h$ from $\mathbb{D}$ onto $\Omega$ satisfying $h(0)=\psi(a,0,0;0).$ Thus, for $p\in\mathcal{H}[a,n],$ \eqref{adm basic} can be written as
\begin{equation}\label{adm h}
	\psi(p(z),zp'(z),z^2p''(z);z)\prec h(z) \ \ \Rightarrow \ \ p(z)\prec q(z).
\end{equation}

	The univalent function $q$ is said to be the \textit{dominant of the solutions} of the second order differential equation \eqref{adm h}. The dominant $\tilde{q}$ that satisfies $\tilde{q}\prec q$ for all the dominants of \eqref{adm h} is said to be the \textit{best dominant} of \eqref{adm h}.
	
	Consider the function $q:\mathbb{D} \to \mathbb{C}$ defined by $q(z)=\sqrt{1+z},\ z\in\mathbb{D}.$ Clearly, the function $q$ is univalent in $\overline{\mathbb{D}}\setminus\{-1\}.$ Thus, $q\in\mathcal{Q}$ with $E(q)=\{-1\}$ and $q(\mathbb{D})=\{w:|w^2-1|<1\}.$ We now define the admissibility conditions for the function $\sqrt{1+z}.$ Denote $\Psi_n[\Omega,\sqrt{1+z}]$ by $\Psi_n[\Omega,\mathcal{L}].$ Further, the case when $\Omega=\Delta=\{w:|w^2-1|<1\,,\operatorname{Re}w>0\},\ \Psi_n[\Omega,\sqrt{1+z}]$ is denoted by $\Psi_n[\mathcal{L}].$
	
	If $|\zeta|=1,$ then \[
	q(\zeta)\in q(\partial\mathbb{D})=\partial q(\mathbb{D})=\{w:|w^2-1|=1\}
	=\left\{\sqrt{2\cos{2\theta}}e^{i\theta}:-\frac{\pi}{4}<\theta<\frac{\pi}{4}\right\}.\]
	Then, for $\zeta= 2\cos{2\theta} e^{2i\theta} -1,$ we have 
	\[	\zeta q'(\zeta)=\frac{1}{2}\left(\sqrt{2\cos 2 \theta} e^{i\theta}-\frac{1}{\sqrt{2\cos 2 \theta} e^{i\theta}} \right)  \quad 
	\text{and}\quad q''(\zeta)=\dfrac{-1}{4(2\cos{2\theta}e^{2i\theta})^{3/2}}\]
	and hence  \[\operatorname{Re} \left(\frac{\zeta q''(\zeta)}{q'(\zeta)}+1 \right)=\operatorname{Re} \left( \frac{e^{-2 i \theta }}{4 \cos 2 \theta}+\frac{1}{2} \right)=\frac{3}{4}.\]
	
	Thus, the condition of admissibility reduces to $\psi(r,s,t;z)\not \in \Omega$ whenever $(r,s,t;z)\in \operatorname{Dom}\psi$ and
	\begin{equation}\label{adm for q}
	\begin{split}
	&r=\sqrt{2\cos 2 \theta}e^{i \theta},\\ &s=\displaystyle{\frac{m}{2}\left(\sqrt{2\cos 2 \theta} e^{i\theta}-\frac{1}{\sqrt{2\cos 2 \theta} e^{i\theta}} \right)=\frac{me^{3i\theta}}{2\sqrt{2\cos2\theta}}},\\ &\displaystyle{\operatorname{Re}\left( \frac{t}{s}+1\right)\geq\frac{3m}{4}}
	\end{split}
	\end{equation}
	where $\theta\in(-\pi/4,\pi/4)$ and $m\geq n\geq 1.$
	
	
	As a particular case of Theorem \ref{main thm}, we have
	\begin{theorem}\label{lem thm}
		Let $p\in \mathcal{H}[1,n]$ with $p(z)\not\equiv 1$ and $n\geq 1.$ Let $\Omega\subset\mathbb{C}$ and $\psi:\mathbb{C}^3\times\mathbb{D}\to\mathbb{C}$ with domain $D$ satisfy \[\psi(r,s,t;z)\not\in\Omega\ \text{ whenever } z\in\mathbb{D},  \]
		for $r=\sqrt{2\cos 2\theta}e^{i\theta},\ s=me^{3i\theta}/(2\sqrt{2\cos2\theta})$ and $\operatorname{Re}(t/s+1)\geq 3m/4$ where $m\geq n\geq1$ and $-\pi/4<\theta<\pi/4.$ For $z\in\mathbb{D},$ if $(p(z),zp'(z),z^2p''(z);z)\in D,$ and $\psi(p(z),zp'(z),z^2p''(z);z)\in \Omega,$ then $p(z)\prec\sqrt{1+z}.$
	\end{theorem}
	The case when $\psi\in\Psi_n[\mathcal{L}]$ with domain $D,$ the above theorem reduces to the case: For $z\in \mathbb{D},$ if $(p(z),zp'(z),z^2p''(z);z)\in D$ and $\psi(p(z),zp'(z),z^2p''(z);z)\prec \sqrt{1+z},$ then $p(z)\prec\sqrt{1+z}.$
	
	We now illustrate the above result for certain $\Omega.$ Throughout $r,s,t$ refer to as mentioned in \eqref{adm for q}.
	\begin{example}
		Let $\displaystyle{\Omega=\{w:|w-1|<1/(2\sqrt{2})\}}$ and define $\psi:\mathbb{C}^3\times\mathbb{D}\to \mathbb{C} $ by $\psi(a,b,c;z)=1+b.$
		For $\psi$ to be in $\Psi[\Omega,\mathcal{L}],$ we must have $\psi(r,s,t;z)\not\in\Omega$ for 
		$z\in\mathbb{D}.$
		Then, $\psi(r,s,t;z)$ is given by
		\begin{align*}
		\psi(r,s,t;z)&=1+\frac{me^{3i\theta}}{2\sqrt{2\cos2\theta}}\\
		\intertext{and therefore we have that}
		|\psi(r,s,t;z)-1|&=\left|\frac{me^{3i\theta}}{2\sqrt{2\cos2\theta}}\right|=\frac{m}{2\sqrt{2\cos2\theta}}\geq\frac{m}{2\sqrt{2}}\geq\frac{1}{2\sqrt{2}}.
		\end{align*}
		Thus, $\psi\in\Psi[\Omega,\mathcal{L}].$
		Hence, whenever $p\in\mathcal{H}_1$ such that $|zp'(z)|< 1/(2\sqrt{2}),$ then $p(z)\prec \sqrt{1+z}.$
	\end{example}
	\begin{example}
		Let $\Omega=\{w:\operatorname{Re}w<1/4\}$ and define $\psi:(\mathbb{C}\setminus\{0\})\times\mathbb{C}^2\times\mathbb{D}\to\mathbb{C}$ by
		$\psi(a,b,c;z)=b/a.$
		For $\psi$ to be in $\Psi[\Omega,\mathcal{L}],$ we must have $\psi(r,s,t;z)\not\in\Omega$ for 
		$z\in\mathbb{D}.$
		Now, consider $\psi(r,s,t;z)$ given by
		\begin{align*}
		\psi(r,s,t;z)&=\frac{s}{r}=\frac{me^{2i\theta}}{4\cos2\theta}.\\
		\intertext{Then}
		\operatorname{Re}\psi(r,s,t;z)&=\frac{m}{4\cos2\theta}\operatorname{Re}(e^{2i\theta})=\frac{m}{4} \geq\frac{1}{4}.
		\end{align*}
		That is $\psi(r,s,t;z)\not\in \Omega.$ Hence, we see that $\psi \in \Psi[\Omega,\mathcal{L}].$
		Therefore, for $p(z)\in \mathcal{H}_1$ if \[\operatorname{Re}\left(\frac{zp'(z)}{p(z)}\right)<\frac{1}{4},\]
		then $p(z)\prec \sqrt{1+z}.$ 		Moreover, the result is sharp as for $p(z)=\sqrt{1+z},$ we have \[ \operatorname{Re}\left(\frac{zp'(z)}{p(z)}\right)= \operatorname{Re}\left(\frac{z}{2(1+z)}\right)\to \frac{1}{4} \text{ as } z\to 1.  \]
		That is $\sqrt{1+z}$ is the best dominant.
	\end{example}
 
	\begin{example}
		Let $\Omega=\{w:|w-1|<1/(4\sqrt{2})\}$ and define $\psi:(\mathbb{C}\setminus\{0\})\times\mathbb{C}^2\times\mathbb{D}\to \mathbb{C} $ by $\psi(a,b,c;z)=1+b/a^2.$
		For $\psi$ to be in $\Psi[\Omega,\mathcal{L}],$ we must have $\psi(r,s,t;z)\not\in\Omega$ for 
		$z\in\mathbb{D}.$
		Then, $\psi(r,s,t;z)$ is given by
		\begin{align*}
		\psi(r,s,t;z)&=1+\frac{me^{i\theta}}{2(2\cos2\theta)^{3/2}}\\
		\intertext{and so}
		|\psi(r,s,t;z)-1|&=\left|\frac{me^{i\theta}}{2(2\cos2\theta)^{3/2}}\right|=\frac{m}{4\sqrt{2}(\cos2\theta)^{3/2}}\geq\frac{m}{4\sqrt{2}}\geq\frac{1}{4\sqrt{2}}.
		\end{align*}
		Thus, $\psi\in\Psi[\Omega,\mathcal{L}].$
		Hence, whenever $p\in\mathcal{H}_1$ such that \[ \left|\frac{zp'(z)}{p^2(z)}\right|< \frac{1}{4\sqrt{2}},  \] then $p(z)\prec \sqrt{1+z}.$
	\end{example}
\section{First Order Differential Subordination}\label{fods}
In case of first order differential subordination, Theorem \ref{lem thm} reduces to:
\begin{theorem}
	Let $p\in \mathcal{H}[1,n]$ with $p(z)\not\equiv 1$ and $n\geq 1.$ Let $\Omega\subset\mathbb{C}$ and $\psi:\mathbb{C}^2\times\mathbb{D}\to\mathbb{C}$ with domain $D$ satisfy \[\psi(r,s;z)\not\in\Omega\ \text{ whenever } z\in\mathbb{D},  \]for $r=\sqrt{2\cos 2\theta}e^{i\theta}$ and $s=me^{3i\theta}/(2\sqrt{2\cos2\theta})$ where $m\geq n\geq1$ and $-\pi/4<\theta<\pi/4.$ For $z\in\mathbb{D},$ if $(p(z),zp'(z);z)\in D$ and $\psi(p(z),zp'(z);z)\in \Omega,$ then $p(z)\prec\sqrt{1+z}.$
\end{theorem}
Likewise for an analytic function $h,$ if $\Omega=h(\mathbb{D}),$ then the above theorem becomes
\[ \psi(p(z),zp'(z);z)\prec h(z) \Rightarrow p(z)\prec \sqrt{1+z}.  \]

Using the above theorem, now some sufficient conditions are determined for $p\in\mathcal{H}_1$ to satisfy $p(z)\prec\sqrt{1+z}$ and hence sufficient conditions are obtained for function $f\in\mathcal{A}$ to belong to the class $\mathcal{SL}.$

Kumar \emph{et al.} \cite{MR3063215} proved that for $\beta>0$ if $p(z)+\beta zp'(z)/p^n(z)\prec\sqrt{1+z}\ (n=0,1,2),$ then $p(z)\prec\sqrt{1+z}.$ Extending this, we obtain lower bound for $\beta$ so that $p(z)\prec\sqrt{1+z}$ whenever $p(z)+\beta zp'(z)/p^n(z)\prec\sqrt{1+z}\ (n=3,4).$
\begin{lemma}\label{extended version}
	Let $p$ be analytic in $\mathbb{D}$ and $p(0)=1$ and $\beta_0=1.1874.$ Let \[p(z)+\frac{\beta zp'(z)}{p^3(z)}\prec\sqrt{1+z} \ (\beta>\beta_0),  \] then \[p(z)\prec\sqrt{1+z}.  \]
\end{lemma}
\begin{proof}
	Let $\beta>0.$ Let $\Delta=\{w:|w^2-1|<1,\operatorname{Re}w>0\}.$ Let $\psi:(\mathbb{C}\setminus\{0\})\times\mathbb{C}\times \mathbb{D}\to \mathbb{C}$ be defined by $\psi(a,b;z)=a+\beta b/a^3.$
	For $\psi$ to be in $\Psi[\mathcal{L}],$ we must have $\psi(r,s;z)\not\in\Delta$ for
	$z\in\mathbb{D}.$
	Then, $\psi(r,s;z)$ is given by
	\begin{align*}
	\psi(r,s;z)&=\sqrt{2\cos2\theta}e^{i\theta}+\frac{\beta m}{8\cos^2 2\theta},
	\intertext{so that}
	|\psi(r,s;z)^2-1|^2&=1+\frac{\beta^2m^2}{32}(4\sec^3 2\theta+2\sec^2 2\theta-\sec^4 2\theta)+\frac{\beta m}{\sqrt{2}}\sec^{3/2}2\theta\cos3\theta\\
	&\quad{}+\frac{\beta^4m^4}{4096}\sec^8 2\theta+\frac{\beta^3m^3}{64\sqrt{2}}\sec^{11/2}2\theta\cos\theta=:g(\theta)
	\end{align*}
	Observe that $g(\theta)=g(-\theta)$ for all $\theta\in(-\pi/4,\pi/4)$ and the second derivative test shows that the minimum of $g$ occurs at $\theta=0$ for $\beta m>1.1874.$
	For $\beta>1.1874,$ we have $\beta m>1.1874.$ Thus, $g(\theta)$ attains its minimum at $\theta=0$ for $\beta>\beta_0.$
	For $\psi\in \Psi[\mathcal{L}],$ we must have $g(\theta)\geq 1$ for every $\theta \in (-\pi/4,\pi/4)$ and since
	\begin{align*}
	\min g(\theta)&=1+\frac{\beta m}{\sqrt{2}}+\frac{5\beta^2m^2}{32}+\frac{\beta^3m^3}{64\sqrt{2}}+\frac{\beta^4m^4}{4096}\\
	&\geq 1+\frac{\beta }{\sqrt{2}}+\frac{5\beta^2}{32}+\frac{\beta^3}{64\sqrt{2}}+\frac{\beta^4}{4096}>1.
	\end{align*}
	Hence for $\beta>\beta_0,\ \psi\in\Psi[\mathcal{L}]$ and therefore, for $p(z)\in \mathcal{H}_1,$ if \[ p(z)+\frac{\beta zp'(z)}{p^3(z)}\prec\sqrt{1+z}\ (\beta>\beta_0),  \]
	we have \[ p(z)\prec\sqrt{1+z}.   \qedhere\]	
\end{proof}
\begin{lemma}
	Let $p$ be analytic in $\mathbb{D}$ and $p(0)=1$ and $\beta_0=3.58095.$ Let \[p(z)+\frac{\beta zp'(z)}{p^4(z)}\prec\sqrt{1+z} \ (\beta>\beta_0),  \] then \[p(z)\prec\sqrt{1+z}.  \]
\end{lemma}
\begin{proof}
	Let $\beta>0.$ Let $\Delta=\{w:|w^2-1|<1,\operatorname{Re}w>0\}.$ Let $\psi:(\mathbb{C}\setminus\{0\})\times\mathbb{C}\times \mathbb{D}\to \mathbb{C}$ be defined by $\psi(a,b;z)=a+\beta b/a^4.$
	For $\psi$ to be in $\Psi[\mathcal{L}],$ we must have $\psi(r,s;z)\not\in\Delta$ for
	$z\in\mathbb{D}.$
	Then, $\psi(r,s;z)$ is given by
	\begin{align*}
	\psi(r,s;z)&=\sqrt{2\cos2\theta}e^{i\theta}+\frac{\beta me^{-i\theta}}{8\cos^2 2\theta\sqrt{2\cos2\theta} },
	\intertext{so that}
	|\psi(r,s;z)^2-1|^2&=1+\beta m(1-\frac{1}{2}\sec^2 2\theta)+\frac{\beta^2m^2}{64}(\sec^4 2\theta+4\sec^2 2\theta)\\
	&\quad{}+\frac{\beta^3m^3}{256}\sec^6 2\theta+\frac{\beta^4m^4}{128^2}\sec^{10}2\theta=:g(\theta)
	\end{align*}
	Observe that $g(\theta)=g(-\theta)$ for all $\theta\in(-\pi/4,\pi/4)$ and the second derivative test shows that $g$ attains its minimum at $\theta=0$ if $\beta m>3.58095.$
	For $\beta>3.58095,$ we have $\beta m> 3.58095.$ Thus, $g(\theta)$ attains its minimum at $\theta=0$ for $\beta>\beta_0.$
	For $\psi\in \Psi[\mathcal{L}],$ we must have $g(\theta)\geq 1$ for every $\theta \in (-\pi/4,\pi/4)$ and since
	\begin{align*}
	\min g(\theta)&=1+\frac{\beta m}{2}+\frac{5\beta^2m^2}{64}+\frac{\beta^3m^3}{256}+\frac{\beta^4m^4}{128^2}\\
	&\geq 1+\frac{\beta}{2}+\frac{5\beta^2}{64}+\frac{\beta^3}{256}+\frac{\beta^4}{128^2}>1.
	\end{align*}
	Hence for $\beta>\beta_0,\ \psi\in\Psi[\mathcal{L}]$ and therefore, for $p(z)\in \mathcal{H}_1,$ if \[ p(z)+\frac{\beta zp'(z)}{p^4(z)}\prec\sqrt{1+z}\ (\beta>\beta_0),  \]
	we have \[ p(z)\prec\sqrt{1+z}.   \qedhere\]	
\end{proof}
On the similar lines, one can find lower bound for $\beta_n$ such that $p(z)+\beta_n zp'(z)/p^n(z)\prec\sqrt{1+z},\ n\in\mathbb{N}$ implies $p(z)\prec\sqrt{1+z}.$

Now, the conditions on $\beta$ and $\gamma$ are discussed so that $p^2(z)+zp'(z)/(\beta p(z)+\gamma)\prec 1+z$ implies $p(z)\prec\sqrt{1+z}.$
\begin{lemma}
	Let $\beta,\gamma>0$ and $p$ be analytic in $\mathbb{D}$ such that $p(0)=1.$ If  \[ p^2(z)+\frac{zp'(z)}{\beta p(z)+\gamma}\prec 1+z, \] then \[ p(z)\prec \sqrt{1+z}.  \]
\end{lemma}
\begin{proof}
	Let $h$ be the analytic function defined on $\mathbb{D}$ by $h(z)=1+z$ and let $\Omega=h(\mathbb{D})=\{w:|w-1|<1\}.$ Let $\psi:(\mathbb{C}\setminus\{-\gamma/\beta\})\times\mathbb{C}\times\mathbb{D}\to\mathbb{C}$ be defined by \[\psi(r,s;z)=r^2+\frac{s}{\beta r+\gamma}.\] For $\psi$ to be in $\Psi[\Omega,\mathcal{L}],$ we must have $\psi(r,s;z)\not\in\Omega$ for 
	$z\in\mathbb{D}.$ Then, $\psi(r,s;z)$ is given by
	\begin{align*}
	\psi(r,s;z)&=2\cos2\theta e^{2i\theta}+\frac{me^{3i\theta}}{(2\sqrt{2\cos2\theta})(\beta\sqrt{2\cos2\theta}e^{i\theta}+\gamma)},
	\intertext{and so}
	|\psi(r,s;z)-1|^2&=\left[\cos\theta+\frac{m\beta\sqrt{2\cos2\theta}\cos\theta+\gamma m}{2\sqrt{2\cos2\theta}d(\theta)}\right]^2\\
	&\quad{}+\left[\sin\theta-\frac{m\beta\sqrt{2\cos2\theta}\sin\theta}{2\sqrt{2\cos2\theta}d(\theta)}\right]^2,
	\intertext{where $d(\theta)=|\beta\sqrt{2\cos2\theta} e^{i\theta}+\gamma|^2=\cos2\theta(2\beta^2+\gamma^2\sec2\theta+2\beta\gamma\sqrt{\sec2\theta+1}).$}
	\intertext{Hence on solving, we get that}
	|\psi(r,s;z)-1|^2&=1+\frac{\beta^2m^2\sec^2 2\theta}{4(2\beta^2+\gamma^2\sec2\theta+2\beta\gamma\sqrt{\sec2\theta+1})^2}\\
	&\quad{}+\frac{\gamma^2m^2\sec^3 2\theta}{8(2\beta^2+\gamma^2\sec2\theta+2\beta\gamma\sqrt{\sec2\theta+1})^2}\\
	&\quad{}+\frac{\beta\gamma m^2\sqrt{\sec2\theta+1}\sec^2 2\theta}{4(2\beta^2+\gamma^2\sec2\theta+2\beta\gamma\sqrt{\sec2\theta+1})^2}\\
	&\quad{}+\frac{\beta m}{2\beta^2+\gamma^2\sec2\theta+2\beta\gamma\sqrt{\sec2\theta+1}}\\
	&\quad{}+\frac{\gamma m \sqrt{\sec2\theta+1}\sec2\theta}{2(2\beta^2+\gamma^2\sec2\theta+2\beta\gamma\sqrt{\sec2\theta+1})}=:g(\theta)
	\end{align*}
	Using the second derivative test, we get that minimum of $g$ occurs at $\theta=0.$
	For $\psi\in\Psi[\Omega,\mathcal{L}],$ we must have $g(\theta)\geq1$ for every $\theta\in(-\pi/4,\pi/4)$ and since
	\begin{align*}
	\min g(\theta)&=1+\frac{\beta^2m^2}{4(\beta\sqrt{2}+\gamma)^4}+\frac{\gamma^2m^2}{8(\beta\sqrt{2}+\gamma)^4}+\frac{\beta\gamma m^2}{2\sqrt{2}(\beta\sqrt{2}+\gamma)^4}\\
	&\quad{}+\frac{\beta m}{(\beta\sqrt{2}+\gamma)^2}+\frac{\gamma m}{\sqrt{2}(\beta\sqrt{2}+\gamma)^2}\\
	&\geq 1+\frac{\beta^2}{4(\beta\sqrt{2}+\gamma)^4}+\frac{\gamma^2}{8(\beta\sqrt{2}+\gamma)^4}+\frac{\beta\gamma}{2\sqrt{2}(\beta\sqrt{2}+\gamma)^4}+\frac{\beta}{(\beta\sqrt{2}+\gamma)^2}\\
	&\quad{}+\frac{\gamma}{\sqrt{2}(\beta\sqrt{2}+\gamma)^2}>1.
	\end{align*}
	Hence, for $\beta,\gamma>0,\ \psi\in\Psi[\Omega,\mathcal{L}]$ and therefore, for $p\in\mathcal{H}_1,$ if \[p^2(z)+\frac{zp'(z)}{\beta p(z)+\gamma}\prec 1+z,\] then \[p(z)\prec\sqrt{1+z}. \qedhere  \]
\end{proof}
Now, conditions on $\beta$ are derived so that $p^2(z)+\beta zp'(z)/p^n(z)\prec 1+z\ (n=-1,0,1,2)$ implies $p(z)\prec\sqrt{1+z}.$
\begin{lemma}
	Let $p$ be analytic in $\mathbb{D}$ with $p(0)=1.$ Let $\beta$ be a complex number such that $\operatorname{Re}\beta>0.$ If \[ p^2(z)+\beta zp'(z)p(z)\prec 1+z,  \] then \[ p(z)\prec\sqrt{1+z}.  \]
\end{lemma}
\begin{proof}
	Let $h$ be the analytic function defined on $\mathbb{D}$ by $h(z)=1+z$ and let $\Omega=h(\mathbb{D})=\{w:|w-1|<1\}.$ Let $\psi:\mathbb{C}^2\times\mathbb{D}\to\mathbb{C}$ be defined by $\psi(a,b;z)=a^2+\beta ab.$ For $\psi$ to be in $\Psi[\Omega,\mathcal{L}],$ we must have $\psi(r,s;z)\not\in\Omega$ for 
	$z\in\mathbb{D}.$ Then, $\psi(r,s;z)$ is given by
	\begin{align*}
	\psi(r,s;z)&=2\cos2\theta e^{2i\theta}+\frac{\beta me^{4i\theta}}{2}
	\intertext{and we see that}
	|\psi(r,s;z)-1|&=\left|1+\frac{m\beta}{2}\right|\geq1+\frac{m\operatorname{Re}\beta}{2}\geq1+\frac{\operatorname{Re}\beta}{2}> 1.
	\end{align*}
	Hence, for $\beta$ such that $\operatorname{Re}\beta>0,\ \psi\in\Psi[\Omega,\mathcal{L}]$ and therefore, for such complex number $\beta$ and for $p\in\mathcal{H}_1,$ if  \[p^2(z)+\beta zp(z)p'(z)\prec 1+z,\] then \[p(z)\prec\sqrt{1+z}. \qedhere \]
\end{proof}
\begin{lemma}
	Let $\beta>0$ and $p$ be analytic in $\mathbb{D}$ with $p(0)=1.$ If \[p^2(z)+\beta zp'(z)\prec 1+z,  \] then \[ p(z)\prec\sqrt{1+z}.  \]
\end{lemma}
\begin{proof}
	Let $h$ be the analytic function defined on $\mathbb{D}$ by $h(z)=1+z$ and let $\Omega=h(\mathbb{D})=\{w:|w-1|<1\}.$ Let $\psi:\mathbb{C}^2\times\mathbb{D}\to\mathbb{C}$ be defined by $\psi(a,b;z)=a^2+\beta b.$ For $\psi$ to be in $\Psi[\Omega,\mathcal{L}],$ we must have $\psi(r,s;z)\not\in\Omega$ for 
	$z\in\mathbb{D}.$ Then, $\psi(r,s;z)$ is given by
	\begin{align*}
	\psi(r,s;z)&=2\cos2\theta e^{2i\theta}+\frac{\beta me^{3i\theta}}{2\sqrt{2\cos2\theta}},
	\intertext{and so}
	|\psi(r,s;z)-1|^2&=1+\frac{\beta^2m^2}{8}\sec2\theta+\frac{\beta m}{2}\sqrt{\sec2\theta+1}\\
	&\geq 1+\frac{\beta^2m^2}{8}+\frac{\beta m}{\sqrt{2}}\geq1+\frac{\beta^2}{8}+\frac{\beta}{\sqrt{2}}>1.
	\end{align*}
	Hence, for $\beta>0,\ \psi\in\Psi[\Omega,\mathcal{L}]$ and therefore, for $p(z)\in\mathcal{H}_1,$ if \[p^2(z)+\beta zp'(z)\prec 1+z,  \] then \[ p(z)\prec\sqrt{1+z}. \qedhere \]
\end{proof}
\begin{lemma}
	Let $\beta>0$ and $p$ be analytic in $\mathbb{D}$ with $p(0)=1.$ If \[p^2(z)+\frac{\beta zp'(z)}{p(z)}\prec 1+z, \] then \[ p(z)\prec\sqrt{1+z}.  \]
\end{lemma}
\begin{proof}
	Let $h$ be the analytic function defined on $\mathbb{D}$ by $h(z)=1+z$ and let $\Omega=h(\mathbb{D})=\{w:|w-1|<1\}.$ Let $\psi:(\mathbb{C}\setminus\{0\})\times\mathbb{C}\times\mathbb{D}\to\mathbb{C}$ be defined by $\psi(a,b;z)=a^2+\beta b/a.$ For $\psi$ to be in $\Psi[\Omega,\mathcal{L}],$ we must have $\psi(r,s;z)\not\in\Omega$ for 
	$z\in\mathbb{D}.$ Then, $\psi(r,s;z)$ is given by
	\begin{align*}
	\psi(r,s;z)&=2\cos2\theta e^{2i\theta}+\frac{\beta me^{2i\theta}}{4\cos2\theta},
	\intertext{and so}
	|\psi(r,s;z)-1|^2&
	=1+\frac{\beta^2m^2}{16\cos^2 2\theta}+\frac{\beta m}{2}
	\geq1+\frac{\beta^2m^2}{16}+\frac{\beta m}{2}\\
	&\geq1+\frac{\beta^2}{16}+\frac{\beta}{2}> 1.
	\end{align*}
	Hence for $\beta>0,\ \psi\in\Psi[\Omega,\mathcal{L}]$ and therefore, for $p(z)\in\mathcal{H}_1,$ if \[p^2(z)+\frac{\beta zp'(z)}{p(z)}\prec 1+z, \] then \[ p(z)\prec\sqrt{1+z}. \qedhere \]
\end{proof}
\begin{lemma}
	Let $\beta_0=2\sqrt{2}.$ Let $p$ be analytic in $\mathbb{D}$ with $p(0)=1.$ If \[p^2(z)+\frac{\beta zp'(z)}{p^2(z)}\prec 1+z \ (\beta>\beta_0), \] then \[ p(z)\prec\sqrt{1+z}.  \]
\end{lemma}
\begin{proof}
	Let $h$ be the analytic function defined on $\mathbb{D}$ by $h(z)=1+z$ and let $\Omega=h(\mathbb{D})=\{w:|w-1|<1\}.$ Let $\psi:(\mathbb{C}\setminus\{0\})\times\mathbb{C}\times\mathbb{D}\to\mathbb{C}$ be defined by $\psi(a,b;z)=a^2+\beta b/a^2.$ For $\psi$ to be in $\Psi[\Omega,\mathcal{L}],$ we must have $\psi(r,s;z)\not\in\Omega$ for 
	$z\in\mathbb{D}.$ Then, $\psi(r,s;z)$ is given by
	\begin{align*}
	\psi(r,s;z)&=2\cos2\theta e^{2i\theta}+\frac{\beta me^{i\theta}}{4\sqrt{2}\cos^{3/2}2\theta},
	\intertext{and so}
	|\psi(r,s;z)-1|^2
	&=1+\frac{\beta^2m^2}{32\cos^3 2\theta}+\frac{\beta m\cos3\theta}{2\sqrt{2}\cos^{3/2} 2\theta}=:g(\theta)
	\end{align*}
	It is clear using the second derivative test that for $\beta m>2\sqrt{2},$ minimum of $g$ occurs at $\theta=0.$
	For $\beta>2\sqrt{2},\ \beta m> 2\sqrt{2}$ which implies that minimum of $ g(\theta)$ is attained at $\theta=0$ for $\beta>\beta_0.$ Hence
	\[	\min g(\theta)=1+\frac{\beta^2m^2}{32}+\frac{\beta m}{2\sqrt{2}}\geq1+\frac{\beta^2}{32}+\frac{\beta}{2\sqrt{2}}>1.\]
	Hence for $\beta>\beta_0,\ \psi\in\Psi[\Omega,\mathcal{L}]$ and therefore, for $p(z)\in\mathcal{H}_1,$ if \[p^2(z)+\frac{\beta zp'(z)}{p^2(z)}\prec 1+z\ (\beta>\beta_0), \] then \[ p(z)\prec\sqrt{1+z}. \qedhere \]
\end{proof}
Next result depicts some sufficient conditions so that $p(z)\prec\sqrt{1+z}$ whenever $p^2(z)+\beta zp'(z)p(z)\prec(2+z)/(2-z).$
\begin{lemma}
	Let $\beta_0=2$ and $p$ be analytic in $\mathbb{D}$ with $p(0)=1.$ If \[ p^2(z)+\beta zp'(z)p(z)\prec \frac{2+z}{2-z}\ \ (\beta\geq\beta_0),  \] then \[ p(z)\prec\sqrt{1+z}.  \] The lower bound $\beta_0$ is best possible.
\end{lemma}
\begin{proof}
	Let $\beta>0.$ Let $h$ be the analytic function defined on $\mathbb{D}$ by $h(z)=(2+z)/(2-z)$ and let $\Omega=h(\mathbb{D})=\{w:|2(w-1)/(w+1)|<1\}.$ Let $\psi:\mathbb{C}^2\times\mathbb{D}\to\mathbb{C}$ be defined by $\psi(a,b;z)=a^2+\beta ab.$ For $\psi$ to be in $\Psi[\Omega,\mathcal{L}],$ we must have $\psi(r,s;z)\not\in\Omega$ for
	$z\in\mathbb{D}.$ Then, $\psi(r,s;z)$ is given by
	\begin{align*}
	\psi(r,s;z)&=2\cos2\theta e^{2i\theta}+\frac{\beta me^{4i\theta}}{2}
	\intertext{then}
	\left|\frac{2(\psi(r,s;z)-1)}{\psi(r,s;z)+1}\right|^2&=\frac{4(1+m\beta/2)^2}{(1+\beta m/2)^2+4+4(1+\beta m/2)\cos4\theta}=:g(\theta)
	\end{align*}
	Using the second derivative test, one can verify that minimum of $g$ occurs at $\theta=0.$
	Thus \[\min g(\theta)=\frac{4(1+\beta m/2)^2}{(1+\beta m/2)^2+4(1+\beta m/2)+4} . \]
	Now, the inequality
	\begin{align*}
	&\frac{4(1+\beta/2)^2}{(1+\beta/2)^2+4(1+\beta/2)+4}\geq 1
	\intertext{holds if}	
	&3\left(1+\frac{\beta}{2}\right)^2-4-4\left(1+\frac{\beta}{2}\right)\geq 0
	\end{align*}
	or equivalently if      $\beta\geq2.$
	
	Since, $m\geq1,\ \beta m\geq2$ implies that \[\frac{4(1+\beta m/2)^2}{(1+\beta m/2)^2+4(1+\beta m/2)+4}\geq1\] and therefore $\left|\dfrac{2(\psi(r,s;z)-1)}{\psi(r,s;z)+1}\right|^2\geq 1$.
	Hence, for $\beta\geq\beta_0,\ \psi\in\Psi[\Omega,\mathcal{L}]$ and for $p\in\mathcal{H}_1,$ if  \[p^2(z)+\beta zp'(z)p(z)\prec \frac{2+z}{2-z}\ (\beta\geq\beta_0),\] then \[p(z)\prec\sqrt{1+z}. \qedhere \]
\end{proof}
\begin{remark}
	All of the above lemmas give a sufficient condition for $f$ in $\mathcal{A}$ to be lemniscate starlike. This can be seen by defining a function $p:\mathbb{D}\to\mathbb{C}$ by $p(z)=zf'(z)/f(z).$ 
\end{remark}
\section{Second Order Differential Subordinations}\label{sods}
This section deals with the case that if there is an analytic function $p$ such that $p(0)=1$ satisfying a second order differential subordination then $p(z)$ is \textit{subordinate} to $\sqrt{1+z}.$ Now, for $r,s,t$ as in \eqref{adm for q}, we have $\displaystyle{\operatorname{Re}\left(\frac{t}{s}+1\right)\geq \frac{3m}{4}}$ for $m\geq n\geq 1.$ On simplyfying,
\begin{align}
\operatorname{Re}(te^{-3i\theta})&\geq \frac{m(3m-4)}{8\sqrt{2\cos2\theta}}.\label{re t}
\intertext{If $m\geq2,$ then}\operatorname{Re}(te^{-3i\theta})
&\geq \frac{1}{2\sqrt{2\cos2\theta}}\geq \frac{1}{2\sqrt{2}}.\nonumber
\end{align}
\begin{lemma}\label{lem5.1}
	Let $p$ be analytic in $\mathbb{D}$ such that $p(0)=1.$ If \[ zp'(z)+z^2p''(z)\prec\frac{3z}{8\sqrt{2}},  \] then \[ p(z)\prec\sqrt{1+z}.  \]
\end{lemma}
\begin{proof}
	Let $h(z)=3z/(8\sqrt{2}),$ then $\Omega=h(\mathbb{D})=\{w:|w|<3/(8\sqrt{2})\}$ and let $\psi:\mathbb{C}^3\times\mathbb{D}\to\mathbb{C}$ be defined by $\psi(a,b,c;z)=b+c.$ For $\psi$ to be in $\Psi[\Omega,\mathcal{L}],$ we must have $\psi(r,s,t;z)\not\in\Omega$ for 
	$z\in\mathbb{D}.$	Then, $\psi(r,s,t;z)$ is given by
	\begin{align*}
	\psi(r,s,t;z)&=\frac{me^{3i\theta}}{2\sqrt{2\cos2\theta}}+t.
	\intertext{So, we have that}
	|\psi(r,s,t;z)|&=\left|\frac{m}{2\sqrt{2\cos2\theta}}+te^{-3i\theta}\right|\geq\frac{3m^2}{8\sqrt{2\cos2\theta}}.
	\intertext{Since $m\geq1,$ so}
	|\psi(r,s,t;z)|&\geq \frac{3}{8\sqrt{2\cos2\theta}} \geq \frac{3}{8\sqrt{2}}.
	\end{align*}
	Therefore, $\psi\in\Psi[\Omega,\mathcal{L}].$ Hence, for $p\in\mathcal{H}_1$ if \[ zp'(z)+z^2p''(z)\prec\frac{3z}{8\sqrt{2}},  \] then \[ p(z)\prec\sqrt{1+z}. \qedhere \]
\end{proof}
We obtain the following theorem by taking $p(z)=zf'(z)/f(z)$ in Lemma \ref{lem5.1}, where $p$ is analytic in $\mathbb{D}$ and $p(0)=1.$
\begin{theorem}
	Let $f$ be a function in $\mathcal{A}.$ If $f$ satisfies the subordination
	\begin{align*}
	&\frac{zf'(z)}{f(z)}\left(1+\frac{zf''(z)}{f'(z)}-\frac{zf'(z)}{f(z)}\right)+\frac{zf'(z)}{f(z)}\Bigg(\frac{z^2f'''(z)}{f'(z)}-\frac{3z^2f''(z)}{f(z)}\\
	&\quad{}+\frac{2zf''(z)}{f'(z)}+2\left(\frac{zf'(z)}{f(z)}\right)^2-\frac{2zf'(z)}{f(z)}\Bigg)\prec\frac{3z}{8\sqrt{2}} ,
	\end{align*}
	then $f\in\mathcal{SL}.$
\end{theorem}
\begin{lemma}\label{lem5.3}
	Let $p$ be analytic in $\mathbb{D}$ such that $p(0)=1$ and let $p\in\mathcal{H}[1,2].$ If \[ p^2(z)+zp'(z)+z^2p''(z)\prec1+\left(1+\frac{3}{2\sqrt{2}}\right)z,  \] then \[ p(z)\prec\sqrt{1+z}.  \]
\end{lemma}
\begin{proof}
	Let $h(z)=1+(1+3/(2\sqrt{2}))z$ then $\Omega=h(z)=\{w:|w-1|<1+3/(2\sqrt{2})\}$. Let $\psi:\mathbb{C}^3\times\mathbb{D}\to\mathbb{C}$ be defined by $\psi(a,b,c;z)=a^2+b+c.$ For $\psi$ to be in $\Psi[\Omega,\mathcal{L}],$ we must have $\psi(r,s,t;z)\not\in\Omega$ for 
	$z\in\mathbb{D}.$ Then, $\psi(r,s,t;z)$ is given by
	\begin{align*}
	\psi(r,s,t;z)&=2\cos2\theta e^{2i\theta}+\frac{me^{3i\theta}}{2\sqrt{2\cos2\theta}}+t.\\
	\intertext{So, we have}
	|\psi(r,s,t;z)-1|&=\left|e^{i\theta}+\frac{m}{2\sqrt{2\cos2\theta}}+te^{-3i\theta}\right|\\
	&\geq \operatorname{Re}\left(e^{i\theta}+\frac{m}{2\sqrt{2\cos2\theta}}+te^{-3i\theta}\right)\\
	&=\cos\theta+\frac{3m^2}{8\sqrt{2}}\sec^{1/2}2\theta =: g(\theta)
	\end{align*}
	The second derivative test shows that minimum of $g$ occurs at $\theta=0$ if $m\geq2.$ Therefore, $\psi\in\Psi[\Omega,\mathcal{L}].$ Hence, for $p\in\mathcal{H}[1,2]$ if \[ p^2(z)+zp'(z)+z^2p''(z)\prec1+\left(1+\frac{3}{2\sqrt{2}}\right)z,\] then \[ p(z)\prec\sqrt{1+z}.  \qedhere \]
\end{proof}
The following theorem holds by taking $p(z)=zf'(z)/f(z)$ in Lemma \ref{lem5.3}, where $p$ is analytic in $\mathbb{D}$ and $p(0)=1.$
\begin{theorem}
	Let $f$ be a function in $\mathcal{A}$ such that $zf'(z)/f(z)$ has Taylor series expansion of the form $1+a_2z^2+a_3z^3+\ldots.$ If $f$ satisfies the subordination
	\begin{align*}
	&\left(\frac{zf'(z)}{f(z)}\right)^2+\frac{zf'(z)}{f(z)}\left(1+\frac{zf''(z)}{f'(z)}-\frac{zf'(z)}{f(z)}\right)+\frac{zf'(z)}{f(z)}\Bigg(\frac{z^2f'''(z)}{f'(z)}-\frac{3z^2f''(z)}{f(z)}\\
	&\quad{}+\frac{2zf''(z)}{f'(z)}+2\left(\frac{zf'(z)}{f(z)}\right)^2-\frac{2zf'(z)}{f(z)}\Bigg)\prec1+\left(1+\frac{3}{2\sqrt{2}}\right)z,
	\end{align*}
	then $f\in\mathcal{SL}.$
\end{theorem}
The next result admits some conditions on $\beta$ and $\gamma$ for $p(z)\prec\sqrt{1+z}$ whenever $\gamma z p'(z)+\beta z^2p''(z)\prec z/(8\sqrt{2}).$
\begin{lemma}\label{lem5.5}
	Let $\gamma\geq\beta>0$ be such that $4\gamma-\beta\geq 1.$ Let $p$ be analytic in $\mathbb{D}$ such that $p(0)=1$ and \[ \gamma z p'(z)+\beta z^2p''(z)\prec\frac{z}{8\sqrt{2}} \text{ for }\gamma\geq \beta>0 \text{ and } 4\gamma-\beta\geq 1, \] then\[p(z)\prec\sqrt{1+z}.\]
\end{lemma}
\begin{proof}
	Let $h(z)=z/(8\sqrt{2})$ for $z\in\mathbb{D}$ and $\Omega=h(\mathbb{D})=\{w:|w|<1/(8\sqrt{2})\}.$ Let $\psi:\mathbb{C}^3\times\mathbb{D}\to\mathbb{C}$ be defined by $\psi(a,b,c;z)=\gamma b+\beta c.$ For $\psi$ to be in $\Psi[\Omega,\mathcal{L}],$ we must have $\psi(r,s,t;z)\not\in\Omega$ for 
	$z\in\mathbb{D}.$ Then, $\psi(r,s,t;z)$ is given by
	\begin{align*}
	\psi(r,s,t;z)&=\frac{\gamma me^{3i\theta}}{2\sqrt{2\cos2\theta}}+\beta t.
	\intertext{Hence, we see that}
	|\psi(r,s,t;z)|&=\left|\frac{\gamma m}{2\sqrt{2\cos2\theta}}+\beta te^{-3i\theta}\right|\geq \frac{\gamma m}{2\sqrt{2\cos2\theta}}+\beta\operatorname{Re}(te^{-3i\theta}).
	\intertext{Using \eqref{re t},}|\psi(r,s,t;z)|&\geq \frac{4m(\gamma-\beta)+3\beta m^2}{8\sqrt{2\cos2\theta}}.
	\intertext{Since $m\geq 1,$ so}
	|\psi(r,s,t;z)|&\geq \frac{4(\gamma-\beta)+3\beta}{8\sqrt{2\cos2\theta}}=\frac{4\gamma-\beta}{8\sqrt{2\cos2\theta}}.
	\intertext{Given that $4\gamma-\beta\geq 1,$}
	|\psi(r,s,t;z)|&\geq \frac{1}{8\sqrt{2\cos2\theta}} \geq \frac{1}{8\sqrt{2}}.
	\end{align*}
	Therefore, $\psi\in\Psi[\Omega,\mathcal{L}].$ Hence for $p\in\mathcal{H}_1$ satisfying \[ \gamma z p'(z)+\beta z^2p''(z)\prec\frac{z}{8\sqrt{2}} \text{ for }\gamma\geq \beta>0 \text{ and } 4\gamma-\beta\geq 1, \] we have\[p(z)\prec\sqrt{1+z}.\qedhere\]
\end{proof}
By taking $p(z)=zf'(z)/f(z)$ in Lemma \ref{lem5.5}, where $p$ is analytic in $\mathbb{D}$ and $p(0)=1,$ the following theorem holds.
\begin{theorem}
	Let $f$ be a function in $\mathcal{A}.$ Let $\gamma,\beta$ be as stated in Lemma \ref{lem5.5}. If $f$ satisfies the subordination
	\begin{align*}
	&\gamma \frac{zf'(z)}{f(z)}\left(1+\frac{zf''(z)}{f'(z)}-\frac{zf'(z)}{f(z)}\right)+\beta \frac{zf'(z)}{f(z)}\Bigg(\frac{z^2f'''(z)}{f'(z)}\\
	&\quad{}-\frac{3z^2f''(z)}{f(z)}+\frac{2zf''(z)}{f'(z)}+2\left(\frac{zf'(z)}{f(z)}\right)^2-\frac{2zf'(z)}{f(z)}\Bigg)\prec\frac{z}{8\sqrt{2}}, 
	\end{align*}
	then $f\in\mathcal{SL}.$
\end{theorem}
\section{Further results}\label{al}
Now, we discuss alternate proofs to the results proven in \cite{MR2917253} where lower bounds for $\beta$ are determined for the cases where $1+\beta zp'(z)/p^n(z)\prec\sqrt{1+z}\ (n=0,1,2)$ imply $p(z)\prec\sqrt{1+z}.$ The method of admissible functions provides an improvement over the results proven in \cite{MR2917253}.
\begin{lemma}\label{lem-1-0}
	Let $p$ be analytic function on $\mathbb{D}$ and $p(0)=1.$ Let $\beta_0=2\sqrt{2}(\sqrt{2}-1)\approx 1.17.$ If \[ 1+\beta zp'(z) \prec \sqrt{1+z}\ (\beta\geq\beta_0),  \] then \[ p(z) \prec \sqrt{1+z}.  \]
\end{lemma}
\begin{proof}
	Let $\beta>0.$ Let $\Delta=\{w:|w^2-1|<1,\operatorname{Re}w>0\}.$ Let us define $\psi:\mathbb{C}^2\times \mathbb{D}\to \mathbb{C}$ by $\psi(a,b;z)=1+\beta b.$
	For $\psi$ to be in $\Psi[\mathcal{L}],$ we must have $\psi(r,s;z)\not\in\Delta$ for
	$z\in\mathbb{D}.$
	Then, $\psi(r,s;z)$ is given by
	\begin{align*}
		\psi(r,s;z)&=1+\frac{\beta m}{2\sqrt{2\cos{2\theta}}}e^{3i\theta}
		\intertext{and so}
		|\psi(r,s;z)^2-1|^2
		&=\frac{\beta^4m^4}{64}\sec^2{2\theta}+\frac{\beta^3m^3}{4\sqrt{2}}\sec^{3/2}{2\theta}\cos3\theta+\frac{\beta^2m^2}{2}\sec{2\theta}=:g(\theta)
	\end{align*}
Observe that $g(\theta)=g(-\theta)$ for all $\theta\in(-\pi/4,\pi/4)$ and the second derivative shows that the minimum of $g$ occurs at $\theta=0$ when $\beta>2\sqrt{2}(\sqrt{2}-1).$
For $\psi\in \Psi[\mathcal{L}],$ we must have $g(\theta)\geq 1$ for every $\theta \in (-\pi/4,\pi/4)$ and since
\[
\min g(\theta)=\frac{\beta^4m^4}{64}+\frac{\beta^3m^3}{4\sqrt{2}}+\frac{\beta^2m^2}{2}\geq \frac{\beta^4}{64}+\frac{\beta^3}{4\sqrt{2}}+\frac{\beta^2}{2}.\]
The last term is greater than or equal to 1 if
\begin{align*}
&(\beta+2\sqrt{2})^2(\beta-4+2\sqrt{2})(\beta+4+2\sqrt{2})\geq 0
\intertext{ or equivalently if }
&\beta\geq 4-2\sqrt{2}=2\sqrt{2}(\sqrt{2}-1)=\beta_0.
\end{align*}
Hence, for $\beta\geq\beta_0,\ \psi\in\Psi[\mathcal{L}]$ and therefore for $p(z)\in \mathcal{H}_1,$ if \[ 1+\beta zp'(z) \prec \sqrt{1+z}\ (\beta\geq\beta_0), \] then, we have \[ p(z)\prec\sqrt{1+z}.   \qedhere\]	
\end{proof}
As in \cite[Theorem 2.2]{MR2917253}, using above lemma, we deduce the following.
\begin{theorem}
	Let $\beta_0=2\sqrt{2}(\sqrt{2}-1)\approx1.17$ and $f\in \mathcal{A}.$
	\begin{enumerate}
		\item If $f$ satisfies the subordination \[ 1+\beta\frac{zf'(z)}{f(z)}\left(1+\frac{zf''(z)}{f'(z)}-\frac{zf'(z)}{f(z)}\right)\prec \sqrt{1+z}\ \ (\beta\geq \beta_0), \]
		then $f\in\mathcal{SL}$.
		\item If $1+\beta zf''(z)\prec\sqrt{1+z}\ (\beta\geq\beta_0),$ then $f'(z)\prec\sqrt{1+z}.$
	\end{enumerate}
\end{theorem}
\begin{lemma}\label{lem-1-1}
	Let $p$ be analytic function on $\mathbb{D}$ and $p(0)=1.$ Let $\beta_0=4(\sqrt{2}-1)\approx1.65$. If \[ 1+\beta \frac{zp'(z)}{p(z)}\prec\sqrt{1+z}\ (\beta\geq\beta_0), \] then \[ p(z)\prec \sqrt{1+z}.  \]
\end{lemma}
\begin{proof}
	Let $\beta>0.$ Let $\Delta=\{w:|w^2-1|<1\,, \operatorname{Re} w >0\}.$ Let $\psi:(\mathbb{C}\setminus\{0\})\times \mathbb{C}\times \mathbb{D} \to \mathbb{C}$ be defined by $\psi(a,b;z)=1+\beta b/a.$
	For $\psi$ to be in $\Psi[\mathcal{L}],$ we must have $\psi(r,s;z)\not\in\Delta$ for
	$z\in\mathbb{D}.$ Then, $\psi(r,s;z)$ is given by
	\begin{align*}
		\psi(r,s;z)&=1+\beta \frac{m}{2}\left(1-\frac{e^{-2i\theta}}{2\cos{2\theta}}\right)
		\intertext{so that}
		|\psi(r,s;z)^2-1|^2
		&= \frac{\beta^4m^4}{256}\sec^4{2\theta}+\left(\frac{\beta^2m^2}{4}+\frac{\beta^3m^3}{16}\right)\sec^2{2\theta}\\
		&\geq \frac{\beta^4m^4}{256}+\left(\frac{\beta^2m^2}{4}+\frac{\beta^3m^3}{16}\right)
		\geq\frac{\beta^4}{256}+\frac{\beta^2}{4}+\frac{\beta^3}{16}
		\intertext{The last term is greater than or equal to 1 if}
		&(\beta+4)^2(\beta+4+4\sqrt{2})(\beta+4-4\sqrt{2})\geq 0,
		\end{align*}
		which is same is $\beta \geq 4\sqrt{2}-4=\beta_0.$
	
Therefore, for $p(z)\in\mathcal{H}_1,$ if \[1+\beta \frac{zp'(z)}{p(z)}\prec\sqrt{1+z} \ (\beta\geq\beta_0),  \] we have \[ p(z)\prec\sqrt{1+z}.  \qedhere\]
\end{proof}
As in \cite{MR2917253}, Theorem 2.4, we get the following.
\begin{theorem}
	Let $\beta_0=4(\sqrt{2}-1)\approx1.65$ and $f\in\mathcal{A}.$
	\begin{enumerate}
		\item If $f$ satisfies the subordination \[ 1+\beta\left(1+\frac{zf''(z)}{f'(z)}-\frac{zf'(z)}{f(z)}\right)\prec \sqrt{1+z}\ \ (\beta\geq\beta_0),  \]then $f\in\mathcal{SL}.$
		\item If $1+\beta zf''(z)/f'(z)\prec\sqrt{1+z}\ (\beta\geq\beta_0),$ then $f'(z)\prec\sqrt{1+z}.$
		\item If $f$ satisfies the subordination \[ 1+\beta\left(\frac{(zf(z))''}{f'(z)}-\frac{2zf'(z)}{f(z)}\right)\prec\sqrt{1+z}\ \ (\beta\geq\beta_0), \] then $z^2f'(z)/f^2(z)\prec\sqrt{1+z}.$
	\end{enumerate}
\end{theorem}
\begin{lemma}\label{lem-1-2}
	Let $p$ be analytic function on $\mathbb{D}$ and $p(0)=1.$ Let $\beta_0=4\sqrt{2}(\sqrt{2}-1)\approx2.34.$ If \[ 1+\beta \frac{zp'(z)}{p^2(z)}\prec\sqrt{1+z} \ (\beta\geq\beta_0), \] then \[ p(z)\prec \sqrt{1+z}.  \]
\end{lemma}
\begin{proof}
	Let $\beta>0.$ Let $\Delta=\{w:|w^2-1|<1\,, \operatorname{Re} w >0\}.$ Let $\psi:(\mathbb{C}\setminus\{0\})\times \mathbb{C}\times \mathbb{D} \to \mathbb{C}$ be defined by $\psi(a,b;z)=1+\beta b/a^2.$
	For $\psi$ to be in $\Psi[\mathcal{L}],$ we must have $\psi(r,s;z)\not\in\Delta$ for
	$z\in\mathbb{D}.$ Then, $\psi(r,s;z)$ is given by
	\begin{align*}
\psi(r,s;z)&=1+\beta\frac{me^{i\theta}}{4\sqrt{2}\cos^{3/2}2\theta}
		\intertext{so that}
	|\psi(r,s,t;z)^2-1|^2
		&=\frac{\beta^4m^4}{1024}\sec^6 2\theta+\frac{\beta^2m^2}{8}\sec^3 2\theta+\frac{\beta^3m^3}{64}\sec^4{2\theta}\sqrt{\sec2\theta+1}\\
		&\geq\frac{\beta^4m^4}{1024}+\frac{\beta^2m^2}{8}+\frac{\beta^3m^3}{32\sqrt{2}}\geq\frac{\beta^4}{1024}+\frac{\beta^2}{8}+\frac{\beta^3}{32\sqrt{2}}.
		\end{align*}
		The last term is greater than or equal to 1 if
		\begin{align*}
		&(\beta+4\sqrt{2})^2(\beta-4\sqrt{2}(\sqrt{2}-1))(\beta+4\sqrt{2}(\sqrt{2}+1))\geq 0
		\intertext{equivalently}
		&\beta \geq 4\sqrt{2}(\sqrt{2}-1)=\beta_0.
		\end{align*}
	Thus, for $\beta\geq \beta_0,$ we have $\psi\in\Psi[\mathcal{L}].$
	Therefore, for $p(z)\in\mathcal{H}_1,$ if \[1+\beta \frac{zp'(z)}{p^2(z)}\prec\sqrt{1+z} \ (\beta\geq\beta_0),  \] we have \[ p(z)\prec\sqrt{1+z}.  \qedhere\]
\end{proof}
By taking $p(z)=\dfrac{zf'(z)}{f(z)}$ as in \cite{MR2917253}, we obtain the following.
\begin{theorem}
	Let $\beta_0=4\sqrt{2}(\sqrt{2}-1)\approx2.34$ and $f\in\mathcal{A}.$ If $f$ satisfies the subordination \[ 1-\beta+\beta\left(\frac{1+zf''(z)/f'(z)}{zf'(z)/f(z)}\right)\prec\sqrt{1+z}\ (\beta\geq\beta_0), \] then $f\in\mathcal{SL}.$
\end{theorem}
Kumar \emph{et al.} introduced that for every $\beta>0,\ p(z)\prec\sqrt{1+z}$ whenever $p(z)+\beta zp'(z)/p^n(z)\prec\sqrt{1+z}\ (n=0,1,2).$ Using admissibility conditions \eqref{adm for q}, alternate proofs to the mentioned results are discussed below.
\begin{lemma}\label{lem-4-0}
	Let $\beta>0$ and $p$ be analytic in $\mathbb{D}$ and $p(0)=1$ such that \[p(z)+\beta zp'(z)\prec\sqrt{1+z}, \] then \[p(z)\prec\sqrt{1+z}.  \]
\end{lemma}
\begin{proof}
	Let $\beta>0.$ Let $\Delta=\{w:|w^2-1|<1\,, \operatorname{Re} w >0\}.$ Let $\psi:\mathbb{C}^2\times \mathbb{D} \to \mathbb{C}$ be defined by $\psi(a,b;z)=a+\beta b.$
	For $\psi$ to be in $\Psi[\mathcal{L}],$ we must have $\psi(r,s;z)\not\in\Delta$ for
	$z\in\mathbb{D}.$ Then, $\psi(r,s;z)$ is given by
	\begin{align*}
		\psi(r,s;z)&=\sqrt{2\cos2\theta}e^{i\theta}+\beta\frac{me^{3i\theta}}{2\sqrt{2\cos2\theta}}
		\intertext{so that}
		|\psi(r,s;z)^2-1|^2&=1+2\beta m+\frac{5\beta^2 m^2}{4}+\frac{\beta^3 m^3}{4}+\frac{\beta^4 m^4}{64}\sec^2 2\theta\\
		&\geq 1+2\beta m+\frac{5\beta^2 m^2}{4}+\frac{\beta^3 m^3}{4}+\frac{\beta^4 m^4}{64}\\
		&\geq 1+2\beta+\frac{5\beta^2}{4}+\frac{\beta^3}{4}+\frac{\beta^4}{64}>1.
		\end{align*}
	Thus $\psi\in\Psi[\mathcal{L}].$ Therefore, for $p(z)\in\mathcal{H}_1,$ if \[p(z)+\beta zp'(z)\prec\sqrt{1+z} \ (\beta>0),  \] we have \[ p(z)\prec\sqrt{1+z}.  \qedhere\]
\end{proof}
Taking $p(z)=zf'(z)/f(z)$ and $p(z)=f'(z),$ we get the following.
\begin{theorem}
	Let $\beta>0$ and $f$ be a function in $\mathcal{A}.$
	\begin{enumerate}
		\item If $f$ satisfies the subordination \[\frac{zf'(z)}{f(z)}+\beta\frac{zf'(z)}{f(z)}\left(1+\frac{zf''(z)}{f'(z)}-\frac{zf'(z)}{f(z)}\right)\prec \sqrt{1+z}, \] then $f\in\mathcal{SL}.$
		\item If $f'(z)+\beta zf''(z)\prec\sqrt{1+z},$ then $f'(z)\prec\sqrt{1+z}.$
	\end{enumerate}
\end{theorem}
\begin{lemma}\label{lem-4-1}
	Let $\beta>0$ and $p$ be analytic in $\mathbb{D}$ and $p(0)=1$ such that \[p(z)+\frac{\beta zp'(z)}{p(z)}\prec\sqrt{1+z}, \] then \[p(z)\prec\sqrt{1+z}.  \]
\end{lemma}
\begin{proof}
	Let $\beta>0.$ Let $\Delta=\{w:|w^2-1|<1\,, \operatorname{Re} w >0\}.$ Let $\psi:(\mathbb{C}\setminus\{0\})\times\mathbb{C}\times \mathbb{D} \to \mathbb{C}$ be defined by $\psi(a,b;z)=a+\beta b/a.$
	For $\psi$ to be in $\Psi[\mathcal{L}],$ we must have $\psi(r,s;z)\not\in\Delta$ for
	$z\in\mathbb{D}.$ Then, $\psi(r,s;z)$ is given by
	\begin{align*}
		\psi(r,s;z)&=\sqrt{2\cos2\theta}e^{i\theta}+\beta\frac{me^{2i\theta}}{4\cos2\theta}
		\intertext{so that}
		|\psi(r,s,t;z)^2-1|^2
		&=1+\frac{\beta^4m^4}{256}\sec^4 2\theta+\frac{\beta^2m^2}{8}\sec^2 2\theta+\frac{\beta^2m^2}{2}\sec2\theta\\
		&\quad{}+\beta m\sqrt{\sec2\theta+1}+\frac{\beta^3m^3}{16}\sqrt{\sec2\theta+1}\sec^2\theta\\
		&\geq 1+\sqrt{2}\beta m+\frac{5\beta^2m^2}{8}+\frac{\beta^3m^3}{8\sqrt{2}}+\frac{\beta^4m^4}{256}\\
		&\geq 1+\sqrt{2}\beta+\frac{5\beta^2}{8} +\frac{\beta^3}{8\sqrt{2}}+\frac{\beta^4}{256}> 1.
	\end{align*}
	Thus, $\psi\in\Psi[\mathcal{L}].$ Therefore, for $p(z)\in\mathcal{H}_1,$ if \[p(z)+\beta \frac{zp'(z)}{p(z)}\prec\sqrt{1+z} \ (\beta>0),  \] we have \[ p(z)\prec\sqrt{1+z}.  \qedhere\]
\end{proof}
For $p(z)=zf'(z)/f(z)$ and $p(z)=z^2f'(z)/f^2(z),$ we have
\begin{theorem}
	Let $\beta>0$ and $f$ be a function in $\mathcal{A}.$
	\begin{enumerate}
		\item If $f$ satisfies the subordination \[\frac{zf'(z)}{f(z)}+\beta\left(1+\frac{zf''(z)}{f'(z)}-\frac{zf'(z)}{f(z)}\right)\prec \sqrt{1+z}, \] then $f\in\mathcal{SL}.$
		\item If $f$ satisfies the subordination \[ \frac{z^2f'(z)}{f^2(z)}+\beta\left(\frac{(zf(z))''}{f'(z)}-\frac{2zf'(z)}{f(z)}\right)\prec\sqrt{1+z}, \] then $z^2f'(z)/f^2(z)\prec\sqrt{1+z}.$
	\end{enumerate}
\end{theorem}
\begin{lemma}\label{lem-4-2}
	Let $\beta>0$ and $p$ be analytic in $\mathbb{D}$ and $p(0)=1$ such that \[p(z)+\frac{\beta zp'(z)}{p^2(z)}\prec\sqrt{1+z}, \] then \[p(z)\prec\sqrt{1+z}.  \]
\end{lemma}
\begin{proof}
	Let $\beta>0.$ Let $\Delta=\{w:|w^2-1|<1\,, \operatorname{Re} w >0\}.$ Let $\psi:(\mathbb{C}\setminus\{0\})\times\mathbb{C}\times \mathbb{D} \to \mathbb{C}$ be defined by $\psi(a,b;z)=a+\beta b/a^2.$
	For $\psi$ to be in $\Psi[\mathcal{L}],$ we must have $\psi(r,s;z)\not\in\Delta$ for
	$z\in\mathbb{D}.$ Then, $\psi(r,s;z)$ is given by
	\begin{align*}
		\psi(r,s;z)&=\sqrt{2\cos2\theta}e^{i\theta}+\beta\frac{me^{i\theta}}{4\sqrt{2}\cos^{3/2}2\theta},
		\intertext{so that}
		|\psi(r,s;z)^2-1|^2&=1+\beta m+\frac{5\beta^2m^2}{16}\sec^2 2\theta+\frac{\beta^3 m^3}{32}\sec^4 2\theta+\frac{\beta^4m^4}{1024}\sec^6 2\theta\\
		&\geq 1+\beta m+\frac{5\beta^2m^2}{16}+\frac{\beta^3 m^3}{32}+\frac{\beta^4m^4}{1024}\\
		&\geq 1+\beta +\frac{5\beta^2}{16}+\frac{\beta^3}{32}+\frac{\beta^4}{1024}>1.
	\end{align*}
Thus, $\psi\in\Psi[\mathcal{L}].$ Therefore, for $p(z)\in\mathcal{H}_1,$ if \[p(z)+\beta \frac{zp'(z)}{p^2(z)}\prec\sqrt{1+z} \ (\beta>0),  \] we have \[ p(z)\prec\sqrt{1+z}.  \qedhere\]
\end{proof}
Taking $p(z)=\dfrac{zf'(z)}{f(z)},$ we obtain the following.
\begin{theorem}
	Let $\beta>0$ and $f$ be a function in $\mathcal{A}.$ If $f$ satisfies the subordination \[ \frac{zf'(z)}{f(z)}-\beta+\beta\left(\frac{1+zf''(z)/f'(z)}{zf'(z)/f(z)}\right)\prec\sqrt{1+z}, \] then $f\in\mathcal{SL}.$
\end{theorem}

\section*{Acknowledgements}
The first author is supported by University Grants Commission(UGC), UGC-Ref. No.:1069/(CSIR-UGC NET DEC, 2016).

\end{document}